\documentclass[12pt]{amsart}
\usepackage{amssymb,amsmath,xcolor}
\def\H {{\mathcal H}}
\def\M {{\mathcal M}}
\def\B {{\mathbb B}}
\def\Q {{\mathcal Q}}
\def\E {{\mathcal E}}
\def\R {\mathbb{R}}

\def\D {{\mathfrak D}}
\def\HH{{\boldsymbol H}}
\def\VV{{\boldsymbol V}}
\def\VVp{{\boldsymbol{V^\prime}}}
\def\WW{{\boldsymbol W}}
\def\e{{\rm e}}
\def\d{{\rm d}}
\def\ddt{\frac{\d}{\d t}}
\def \l {\langle}
\def \r {\rangle}
\newtheorem{proposition}{Proposition}[section]
\newtheorem{theorem}[proposition]{Theorem}
\newtheorem{corollary}[proposition]{Corollary}
\newtheorem{lemma}[proposition]{Lemma}
\theoremstyle{definition}
\newtheorem{definition}[proposition]{Definition}
\newtheorem{remark}[proposition]{Remark}

\numberwithin{equation}{section}
\def \au {\rm}
\def \ti {\it}
\def \jou {\rm}
\def \bk {\it}
\def \no#1#2#3 {{\bf #1} (#3), #2.}
\def \eds#1#2#3 {#1, #2, #3.}

\title[{NSV equations with memory in 3D}]
{Navier-Stokes-Voigt equations with memory in 3D\\
lacking instantaneous kinematic viscosity}

\author[F. Di Plinio, A. Giorgini, V. Pata and R. Temam]
{Francesco Di Plinio, Andrea Giorgini, Vittorino Pata and Roger Temam}

\address{University of Virginia - Department of Mathematics
\newline\indent
Kerchof Hall, Charlottesville, VA 22904-4137, USA}
\email{francesco.diplinio@virginia.edu {\rm (F. Di Plinio)}}

\address{Politecnico di Milano - Dipartimento di Matematica
\newline\indent
Via Bonardi 9, 20133 Milano, Italy}
\email{andrea.giorgini@polimi.it {\rm (A. Giorgini)}}
\email{vittorino.pata@polimi.it {\rm (V. Pata)}}

\address{Indiana University - Institute for Scientific Computing and Applied Mathematics
\newline\indent
Rawles Hall, Bloomington, IN 47405, USA}
\email{temam@indiana.edu {\rm (R. Temam)}}


\begin{document}

\begin{abstract}
We consider a Navier-Stokes-Voigt fluid model where the instantaneous
kinematic viscosity has been completely replaced by a memory term
incorporating hereditary effects, in presence of Ekman damping.
The dissipative character of our model is  weaker than the one where
hereditary and instantaneous viscosity coexist, studied in \cite{GTM}
by Gal and Tachim-Medjo. Nevertheless, we prove the existence of a
regular exponential attractor of finite fractal dimension under
rather sharp assumptions on the memory kernel.
\end{abstract}

\maketitle

\section{Introduction}
\label{secI}

\noindent
Let $\Omega\subset \mathbb{R}^3$ be a bounded domain with smooth boundary $\partial \Omega$,
and let $\alpha>0$ and $\beta\geq 0$ be given constants. For $t>0$,
we consider the dimensionless form of
the Navier-Stokes-Voigt (NSV) equations with memory and Ekman damping
in the unknown velocity $u=u(x,t)$ and pressure $\pi=\pi(x,t)$
\begin{equation}
\label{NSV}
\begin{cases}
\partial_t  \left(u - \alpha \Delta u\right) - \displaystyle \int_0^\infty  g(s) \Delta u(t-s)\d s
+ \beta u + (u\cdot \nabla) u  +\nabla \pi=f, \\
\text{div}\, u=0,
\end{cases}
\end{equation}
where $f$ is an assigned external forcing term,
subject to the {\it nonslip} boundary condition
\begin{equation}
\label{bc}
\quad u_{|\partial\Omega}=0.
\end{equation}
The function $g:[0,\infty)\to\R$, usually called memory kernel,
is supposed to be convex nonnegative, smooth on $\R^+=(0,\infty)$, vanishing at infinity
and satisfying the relation
$$\int_0^\infty g(s)\d s=1.$$
The system is supplemented with
the initial conditions
\begin{equation}
\label{ic}
u(0)=u_0 \qquad \text{and} \qquad u(-s)_{|s>0}=\varphi_0(s),
\end{equation}
where $u_0$ and $\varphi_0(\cdot)$ are prescribed data.
The initial velocity $u_0$ and the past
history of the velocity $\varphi_0$, which need only be defined for almost every $s>0$,
will play a different role in the translation of~\eqref{NSV} into an evolution system,
as detailed in the subsequent Section~\ref{secDS}.

The drag/friction term $\beta u$, widely used in geophysical hydrodynamics,
is known as Ekman damping. The coefficient $\beta$ is the Ekman pumping/dissipation constant,
induced by the so-called Ekman layer, appearing at the bottom of a rotating fluid (see \cite{PED}).
The term $\beta u$ is also referred to as the Rayleigh friction,
employed in oceanic models such as
the viscous Charney-Stommel barotropic ocean circulation model of the gulf stream.
In the last decades, in connection with
damped Euler and Navier-Stokes equations, several works involving a dissipation term of such kind
have been made, in order to study the limit case
of vanishing viscosity (see e.g.\ \cite{CVZ,CR,IMT}).

Coming back to the memory kernel, the formal choice of $g$
equal to the Dirac mass at $0^+$ in~\eqref{NSV} corresponds to the damped
NSV system
\begin{equation}
\label{NSVinst}
\begin{cases}
\partial_t  \left(u - \alpha \Delta u\right) -   \Delta u
+ \beta u + (u\cdot \nabla) u  +\nabla \pi=f, \\
\text{div}\, u=0.
\end{cases}
\end{equation}
This system was introduced by Oskolkov in \cite{OSKO}, without the friction term $\beta u$, as a
model for the motion of a Kelvin-Voigt linear viscoelastic incompressible fluid.
Indeed, global well-posedness of \eqref{NSVinst},
along with the existence of a finite dimensional global attractor, has been established in connection
with numerical simulations of turbulent flows in statistical equilibrium, see \cite{KT} and
references therein. Letting the viscoelastic length scale $\alpha$ go
to zero we recover the
three-dimensional Navier-Stokes system. The recent contribution \cite{CZG} improves on
the results of \cite{KT} by proving the existence of a family of exponential attractors
parameterized by the length scale $\alpha$ which converges to the (weak) exponential attractor
of the Navier-Stokes system as $\alpha \to 0.$
We are led to the study of \eqref{NSV} as a generalization of \eqref{NSVinst} where the instantaneous
kinematic viscosity of the Kelvin-Voigt fluid has been completely replaced by the physically
more realistic integrodifferential (memory) term, rendering the diffusive effects dependent on the
past history of the velocity. A rigorous derivation of~\eqref{NSV} from the
balance of linear momentum and the constitutive equations of a Kelvin-Voigt fluid can be found
in \cite{GTM}, to which we send for a more comprehensive description of
the physical motivations of the class of fluid models including \eqref{NSV}.

The authors of \cite{GTM} then specialize to the study of the well-posedness
and the asymptotic properties of
the system
$$
\begin{cases}
\partial_t  \left(u - \alpha \Delta u\right) -\nu\Delta u-(1-\nu)\displaystyle \int_0^\infty  g(s) \Delta u(t-s)\d s
+ (u\cdot \nabla) u  +\nabla \pi=f, \\
\text{div}\, u=0,
\end{cases}
$$
for some parameter $\nu\in(0,1)$. This is a version
of \eqref{NSV} with $\beta=0$, where instantaneous and memory-type viscosities coexist.
The main result of \cite{GTM} is the existence of a finite-dimensional exponential attractor
for the related dynamical system.
In addition, replacing the memory kernel $g$ with a family of suitable rescalings $g_\varepsilon$
converging to the Dirac
mass (see point III of Section \ref{secFD}), the corresponding family of
$\varepsilon$-exponential attractors is proved to be robust in the limit
$\varepsilon\to 0$, where the (undamped) system \eqref{NSVinst} is recovered.
Such a coexistence of instantaneous and memory-type viscous terms has been previously exploited
in \cite{Jeffreys}, where the asymptotic behavior of a two-dimensional Jeffreys model is shown to
converge to its Navier-Stokes (formal) limit.

Dealing with \eqref{NSV} is substantially more challenging from a mathematical viewpoint, as the
rather strong dissipative effects of the instantaneous diffusive term $-\Delta u$ exploited in
\cite{GTM} are not available.
The upshot of the present work is that the asymptotic features of the Navier-Stokes-Voigt
system with memory are preserved also in absence of instantaneous kinematic viscosity, provided that
the much weaker, but still physically interesting, Ekman damping $\beta u$ is taken into account.

We now move onto the task of describing our results, as well as the crucial points in our analysis,
in more detail. At the same time, we provide an overview of the structure of our article.
After recalling the relevant mathematical framework for Navier-Stokes type systems with memory
in Section \ref{secMSN}, \eqref{NSV} is recast in Section \ref{secDS} as an abstract evolution
system in the Dafermos history space setting~\cite{DAF}.
Therein, we prove that such an evolution system generates
a strongly continuous semigroup with continuous dependence on the initial data.
The more challenging part of our entire analysis is found in Section \ref{secD}, where we establish
the dissipative character of the semigroup. This is done by combining ad-hoc energy functionals with
a Gronwall-type lemma with parameter borrowed from \cite{Patonw}, and substantially relying on
the Ekman damping term $\beta u$. If no forcing term is present, the dissipative estimate remains
true even if $\beta=0$. This is the content of Section \ref{secEDS}. In the presence of memory terms,
which are intrinsically hyperbolic, finite time regularization of the solution cannot occur.
As a substitute, in Section \ref{secREAS} we show that the trajectories are exponentially attracted
by more regular bounded sets. This property is relied upon in Section \ref{secEA}, where the main
result of the paper, namely, the existence of a regular exponential attractor for the semigroup
of finite fractal dimension in the phase space, is established. The global attractor is thus
recovered as a byproduct. Precise statements are given in  Theorem \ref{EAVoigt} and Corollary
\ref{GAVoigt}, and the proof of the main theorem is carried out in Section \ref{secPT}.
Finally, we send to Section \ref{secFD} for several directions of future investigation.

\section{Mathematical Setting and Notation}
\label{secMSN}

\noindent
Throughout this work, $L^2(\Omega)$, $H_0^1(\Omega)$, $H^r(\Omega)$ will be the standard Lebesgue-Sobolev spaces
on $\Omega$.
In particular, the external force $f$ will be assumed to be a (constant-in-time) vector of $[L^2(\Omega)]^3$.

We introduce the Hilbert space $(\HH,\l \cdot,\cdot \r,\|\cdot\|)$ given by
$$
\HH=\lbrace u \in [L^2(\Omega)]^3: \text{div}\, u=0, \ u\cdot n_{\vert\partial \Omega}=0 \rbrace,
$$
$n$ being the outward normal to $\partial \Omega$,
along with the Hilbert space
$$
\VV=\big\{u\in [H_0^1(\Omega)]^3: \text{div}\, u=0 \big\},
$$
with inner product and norm
$$
\l u,v\r_1= \l\nabla u,\nabla v\r \qquad \text{and} \qquad \|u\|_{1}=\| \nabla u\|.
$$
We denote by $\VVp$ its dual space. We will also encounter the space
$$
\WW=\VV\cap [H^2(\Omega)]^3.
$$
Calling $P: [L^2(\Omega)]^3 \rightarrow \HH$ the Leray orthogonal
projection onto $\HH$, we consider the Stokes operator on $\HH$
$$
A= -P \Delta\quad\text{with domain}\quad \D(A)=\WW.
$$
It is well known that $A$ is a positive selfadjoint operator with compact inverse
(see e.g.\ \cite{TEM}).
This allows us to define the scale of compactly nested Hilbert spaces
$$V^r=\D(A^{\frac{r}2}),\quad r\in\R,$$
endowed with the inner products and norms
$$
\l u, v \r_{r}=\l A^{\frac{r}2}u, A^{\frac{r}2} v \r\qquad \text{and} \qquad \|u\|_{r}=\| A^{\frac{r}2} u\| .
$$
In particular,
$$
V^{-1}=\VVp,\qquad V^0=\HH, \qquad V^1=\VV, \qquad  V^2=\WW.
$$
More generally (see \cite{BF}),
\begin{align*}
V^r & =\VV\cap [H^r(\Omega)]^3, \quad 1\leq r\leq 2,\\
V^r &\subset \VV\cap [H^r(\Omega)]^3, \quad r>2.
\end{align*}
We also recall the Poincar\'{e} inequality
\begin{equation}
\label{Poincare}
\sqrt{\lambda_1} \|u\|\leq  \| u\|_1, \quad \forall \, u \in \VV,
\end{equation}
where $\lambda_1>0$ is the first eigenvalue of $A$.
The symbol $\l \cdot,\cdot \r$ will also stand for the duality product between $V^r$ and its
dual space $V^{-r}$.

As customary, we write
the trilinear form on $\VV\times\VV\times \VV$
$$
b(u,v,w)=\int_{\Omega} (u \cdot \nabla) v \cdot w \d x= \sum_{i,j=1}^3 \int_{\Omega} u_i \frac{\partial v_j}{\partial x_i} w_j  \d x,
$$
satisfying the relation
$$
b(u,v,v)=0.
$$
The
associated bilinear form $B: \VV\times \VV \to \VVp$ is defined as
$$
\l B(u,v),w\r= b(u,v,w).
$$

Next, we turn our attention to the memory kernel. Defining
$$\mu(s)=-g'(s),$$
the {\it prime} being the derivative, we suppose $\mu$
nonnegative, absolutely continuous, decreasing (hence $\mu'\leq 0$ almost everywhere),
and summable on $\R^+$ with total mass
$$\kappa=\int_0^\infty\mu(s)\ \d s>0.$$
The classical Dafermos condition will be also assumed (see \cite{DAF}), namely,
\begin{equation}
\label{dafermos}
\mu'(s)+\delta\mu(s)\leq 0,
\end{equation}
for some $\delta>0$ and almost every $s>0$.
Then, we introduce the $L^2$-weighted Hilbert space on $\R^+$
$$
\M=L^2_{\mu}(\R^+;\VV),
$$
with inner product and norm
$$
\l \eta, \xi \r_{\M}=\int_0^\infty \mu(s)\l \eta(s),\xi(s)\r_{1} \d s
\qquad \text{and} \qquad
\|\eta\|_{\M}=\bigg(\int_0^\infty \mu(s)\|\eta(s)\|_{1}^2\d s\bigg)^\frac12,
$$
along with the infinitesimal generator of the right-translation semigroup on $\M$
$$T\eta=-\partial_s \eta
\quad\text{with domain}\quad
\D(T)=\big\{\eta\in\M:\,
\partial_s \eta\in \M,\,\,\eta(0)=0\big\}.$$
Here, $\partial_s \eta$ is the distributional derivative of $\eta(s)$
with respect to the internal variable $s$.
Finally, we define the {\it extended memory space}
$$\H=\VV\times \M$$
endowed with the product norm
$$\|(u,\eta)\|_{\H}^2=\alpha \|u\|^2_{1}+\|u\|^2
+\|\eta\|^2_{\M}.
$$
Due to the Poincar\'e inequality~\eqref{Poincare}, this norm is equivalent to the natural one
on $\H$.

In this work, we will also make use of higher order memory
spaces. To this end, we define
$$
\M^1=L^2_{\mu}(\R^+;\WW),
$$
with inner product and norm analogous to those of $\M$,
and the corresponding higher order extended memory space
$$
\H^1=\WW\times \M^1
$$
with norm
$$\|(u,\eta)\|_{\H^1}^2=\alpha \|u\|_2^2+\|u\|_1^2
+\|\eta\|^2_{\M^1}.
$$

\subsection*{General agreement}
Throughout the paper, the symbols $C>0$ and $\Q(\cdot)$
will denote a {\it generic} constant and a {\it generic} increasing positive function, respectively,
depending only on the structural parameters of the problem, but
independent of $f$ (unless otherwise specified). Moreover, given a Banach space $\mathcal{X}$ and $R>0$, we denote
by
$$
\mathbb{B}_{\mathcal{X}}(R)=\lbrace x\in \mathcal{X}: \|x\|_{\mathcal{X}}\leq R \rbrace
$$
the ball of $\mathcal{X}$ of radius $R$ about zero.

\section{The Dynamical System}
\label{secDS}

\noindent
As anticipated in the Introduction,
the original problem \eqref{NSV}-\eqref{bc} is
translated
into an evolution system in the so-called Dafermos past history
framework~\cite{DAF}.
First, we observe that, without any loss of generality, we can assume
that $f\in \HH$ which amounts to changing $\pi$
(by adding $Pf-f$ to $\nabla \pi$).
Hence we apply the projection $P$ to the first equation \eqref{NSV} and
we transform \eqref{NSV}-\eqref{bc} into
\begin{equation}
\label{mezzo}
\partial _t\left(u + \alpha Au\right)
+\int_0^\infty  g(s) A u(t-s)\d s + \beta  u
+ B(u,u)=f.
\end{equation}
In more generality, we will replace
the damping $\beta u$ with a term of
the form $\beta A^{-\vartheta}u$,
for some $\vartheta\geq 0$.
Then, we introduce the {\it past history variable}
$$
\eta^t(s)=\int_0^s u(t-\sigma)\d \sigma,
$$
that satisfies the differential identity
$$\partial_t\eta^t(s) = -\partial_s\eta^t(s) + u(t).$$
At this point, recalling the definition of the operator $T$,
a formal integration by parts leads to the differential problem
in the unknown variables
$u=u(t)$ and $\eta=\eta^t(\cdot)$
\begin{equation}
\label{PROBLEM}
\begin{cases}
\partial _t\left(u + \alpha Au\right)
+ \displaystyle \int_0^\infty  \mu(s) A \eta(s)\d s + \beta A^{-\vartheta} u
+ B(u,u)=f,  \\
\partial_t\eta = T \eta + u,
\end{cases}
\end{equation}
for some $\vartheta\geq 0$, where $\mu=-g'$.
As we said, this is actually a generalization of
equation~\eqref{mezzo}, which
corresponds to the case $\vartheta=0$ (the strongest damping within this class).
In turn, the initial conditions \eqref{ic} transform into
$$u(0)=u_0 \qquad \text{and} \qquad \eta^0=\eta_0,
$$
having set
$$\eta_0(s)=\int_0^s \varphi_0(\sigma)\d \sigma.
$$

We begin with the definition of
weak solution.

\begin{definition}
Given $U_0=(u_0,\eta_0)\in \H$, a function $U=(u,\eta)\in \mathcal{C}([0,\infty), \H)$,
with $\partial_t u\in L^2(0,\tau;\VV)$ for every $\tau>0$, is a weak solution
to~\eqref{PROBLEM} with initial datum
$$U(0)=(u(0),\eta^0)=U_0$$
if for every test function $v\in\VV$ and almost every $t>0$
$$\l \partial_t u,v \r+ \alpha \l \partial_t u,v \r_1+
\int_0^\infty  \mu(s) \l \eta(s), v \r_1  \d s +\beta \l A^{-\vartheta}u, v \r
+ b(u,u,v)=\l f,v \r,
$$
where $\eta$
fulfills the explicit representation
\begin{equation}
\label{REP}
\eta^t(s)=
\begin{cases}
\displaystyle \int_0^s u(t-\sigma) \d \sigma, & 0<s\le t ,\\
\eta^0(s-t)+\displaystyle \int_0^t u(t-\sigma) \d \sigma, & s>t.
\end{cases}
\end{equation}
\end{definition}

\begin{remark}
Given $u\in \mathcal{C}([0,\infty), \VV)$, the function
$\eta=\eta^t(\cdot)$ satisfies the representation formula above
if and only if it is a mild solution (in the sense of \cite{PAZ}) to the nonhomogeneous linear equation
$$
\partial_t\eta= T \eta + u.
$$
\end{remark}

\begin{theorem}
For every initial datum $U_0\in \H$,
system \eqref{PROBLEM} admits a unique solution $U(t)=(u(t),\eta^t)$.
\end{theorem}

Accordingly, the problem generates a dynamical system, otherwise called strongly continuous semigroup,
$$S(t):\H\to\H,\quad t\geq 0,$$
acting by the formula
$$
S(t)U_0=U(t).
$$
This is a one-parameter family of maps $S(t)$ on $\H$ satisfying the properties:
\begin{enumerate}
\item[$\diamond$] $S(0)={\rm Id}_\H$;
\item[$\diamond$] $S(t+\tau)=S(t)S(\tau)$, for every $t,\tau\geq 0$;
\item[$\diamond$] $t\mapsto S(t)U_0\in {\mathcal C}([0,\infty),\H)$, for every $U_0\in\H$.
\end{enumerate}
\smallskip
Given an initial datum $U_0\in\H$, the corresponding energy at time $t\geq 0$ reads
$$\E(t)=\frac12\|S(t)U_0\|_\H^2=\frac12\Big[\|u(t)\|^2+\alpha\|u(t)\|_1^2 +\|\eta^t\|^2_{\M}\Big].
$$

The existence of a weak solution is carried out via a
Galerkin approximation scheme, by exploiting standard energy estimates
(much more immediate than the uniform ones of the next Section \ref{secD}),
and then passing to the limit in the usual way.
Indeed (see the energy equality~\eqref{diffeq1}), we easily get that, for
any given $\tau> 0$ and $R\geq 0$,
\begin{equation}
\label{bounduniq}
\|U(t)\|_{\H}\leq  \Q(R+\tau),
\end{equation}
for all $t\leq\tau$ and all initial data $U(0)\in\B_{\H}(R)$.
In particular, the increasing function $\Q$ does not depend on $\beta\geq 0$.
We refer the interested reader to \cite{TIMEX} for more details on the Galerkin
scheme in connection with equations with memory.

\begin{remark}
It is worth noting that the argument does not require $\beta$ to be strictly positive,
nor inequality \eqref{dafermos} is needed at this level. Indeed, both $\beta>0$
and \eqref{dafermos} will come into play in connection with the dissipative properties
of the semigroup.
\end{remark}

Uniqueness is a consequence of the following continuous dependence result.

\begin{proposition}
\label{crtn}
For every $\tau> 0$ and $R\geq 0$, any two solutions $U_1(t)$ and $U_2(t)$ to
equation \eqref{PROBLEM} fulfill the estimate
\begin{equation}
\| U_1(t)-U_2(t)\|_{\H}\leq \Q(R+\tau)\| U_1(0)-U_2(0)\|_{\H},
\end{equation}
for all $t\leq\tau$ and all initial data $U_i(0)\in\B_{\H}(R)$.
\end{proposition}

\begin{proof}
Let $\tau>0$ be fixed,
and let $U_1=(u_1,\eta_1)$, $U_2=(u_2,\eta_2)$ be two solutions to \eqref{PROBLEM} on the time-interval
$[0,\tau]$ such that
$U_i(0)\in\B_{\H}(R)$.
Then, the difference
$$\bar U=(\bar u,\bar\eta)=U_1-U_2$$
solves
$$
\partial _t\left(\bar u + \alpha A\bar u\right)
+ \displaystyle \int_0^\infty  \mu(s) A \bar\eta(s)\d s + \beta A^{-\vartheta} \bar u
+ B(u_1,\bar u)+B(\bar u,u_2)=0.
$$
Testing the equation by the admissible test function $\bar u$, we easily find
$$
\ddt \Gamma  + \l \bar\eta, \bar u\r_{\M} \leq - b(\bar u,u_2,\bar u),
$$
where
$$
\Gamma(t) =  \frac12 \| \bar u(t)\|^2 + \frac{\alpha}{2}\| \bar u(t)\|_1^2.
$$
In light of \eqref{bounduniq},
$$- b(\bar u,u_2,\bar u)\leq C \|u_2\|_1\|\bar u\|_1^2\leq \Q(R+\tau)\Gamma,
$$
and an integration in time yields
$$
\Gamma(t) + \int_0^t \l \bar\eta^y, \bar u(y)\r_{\M} \d y \leq\Gamma(0)+ \Q(R+\tau)\int_0^t \Gamma(y) \d y,
\quad\forall\, t\leq\tau.
$$
Observing that $\bar\eta$ fulfills the representation~\eqref{REP}, from \cite[Theorem 5.1]{TIMEX}
we learn that
$$
\frac12 \| \bar\eta^t\|_{\M}^2-\frac12 \| \bar\eta^0\|_{\M}^2\leq \int_0^t \l \bar\eta^y, \bar u(y)\r_{\M} \d y.
$$
In summary,
$$
\|\bar U(t)\|_{\H}^2\leq \|\bar U(0)\|_{\H}^2+ \Q(R+\tau)\int_0^t \|\bar U(y)\|_{\H}^2 \d y,
$$
and the claim follows from the integral Gronwall lemma.
\end{proof}

In particular, we draw from Proposition~\ref{crtn} that $S(t)$
fulfills the further continuity property
\begin{enumerate}
\item[$\diamond$] $S(t)\in{\mathcal C}(\H,\H)$, for every $t\geq 0$.
\end{enumerate}

\section{Dissipativity}
\label{secD}

\noindent
The main result of this section provides
a uniform-in-time {\it a priori} estimate on the solutions $U(t)=S(t)U_0$.
To this end, we introduce the function
$$
\phi(\vartheta)= \vartheta-\frac12,
$$
and we strengthen the assumptions on the external force
by requiring
\begin{equation}
\label{ASSO}
f\in V^\varrho\qquad\text{for some}\qquad
\varrho>\phi(\vartheta).
\end{equation}
Since $\phi(\vartheta)<\vartheta$, there is no harm in assuming also $\varrho\leq \vartheta$.
When $\phi(\vartheta)<0$, which is the same as saying $\vartheta< 1/2$, we only require $\varrho= 0$,
i.e.\
$$f\in V^0=\HH,$$
which does not add anything to our general assumptions made at the beginning.
As expected, $\phi$ is increasing, since  the influence of
the damping term $A^{-\vartheta}$ becomes weaker when $\vartheta$ is larger.

\begin{remark}
We highlight the fact that $f\in \HH$ in the physically relevant case $\vartheta=0$,
corresponding to the Ekman damping.
\end{remark}

\begin{theorem}
\label{Th-EE}
Let $\beta>0$ be fixed, and let \eqref{ASSO} hold. Then, for any initial datum $U_0\in\H$,
the corresponding energy fulfills the estimate
$$
\E(t)\leq \Q(\E(0))\e^{-\omega t} +C\|f\|^2_\varrho,
$$
where $\omega>0$ is a universal constant depending only on the structural parameters of the problem.
\end{theorem}

\begin{remark}
It is worth mentioning that a dissipative estimate of this kind
seems to be out of reach when $\beta=0$.
\end{remark}

The main consequence of Theorem \ref{Th-EE} is that, when $\beta>0$,
the dynamical system $S(t)$ is dissipative, namely the trajectories originating from any given
bounded set belong uniformly in time to an \textit{absorbing set}.
By definition, this is a bounded set $\mathcal{B}_0\subset\H$ with the following property: for any
bounded set $\mathcal{B}\subset \H$ of initial data
there is an \textit{entering time} $t_\e=t_\e(\mathcal{B})\geq 0$ such that
$$
S(t)\mathcal{B}\subset \mathcal{B}_0,\quad \forall\, t\geq t_\e.
$$
It is then immediate to see  that one can take
$$
\mathcal{B}_0=\B_{\H}(R_0),
$$
for any fixed
$R_0> \sqrt{2C}\| f\|_{\rho}$, with the constant $C$ of the previous statement.

\smallskip
The remaining of the section is devoted to the proof of Theorem \ref{Th-EE}. This will require a number of steps.

\subsection{Technical lemmas}
The main tool needed in the proof is a Gronwall-type lemma from~\cite{Patonw}.

\begin{lemma}
\label{LEMMAPatonwall}
Let $\Lambda_\varepsilon$ be a family of absolutely continuous nonnegative functions
on $[0,\infty)$
satisfying for every $\varepsilon>0$ small and some $\varkappa>0$ and $M\geq 0$
the differential inequality
$$
\ddt\Lambda_\varepsilon+\varkappa\varepsilon \Lambda_\varepsilon
\leq C\varepsilon^p\Lambda_\varepsilon^q+\frac{M}{\varepsilon^r},
$$
where the nonnegative parameters $p,q,r$ fulfill
$$p-1>(q-1)(1+r)\geq 0.$$
Moreover, let $E$ be a continuous nonnegative function on $[0,\infty)$
such that
$$\frac1{m}\,E(t)\leq \Lambda_\varepsilon(t)\leq m E(t)$$
for every $\varepsilon>0$ small and some $m\geq 1$.
Then, there exists $\omega>0$
such that
$$E(t)\leq \Q(E(0))\e^{-\omega t}+CM.$$
\end{lemma}

We also recall a basic interpolation result (see e.g.\ \cite{LM}).

\begin{lemma}
\label{INTERPOLINO}
Let $a<b<c$. Then
$$\|u\|_b\leq \|u\|_c^\varpi \|u\|_a^{1-\varpi},\quad\forall\, u\in V^c,$$
with
$$\varpi=\frac{b-a}{c-a}.$$
\end{lemma}


\subsection{Energy functionals}
Let now
$$U(t)=(u(t), \eta^t)$$
be the solution to \eqref{PROBLEM}
originating from a given $U_0\in\H$.
In what follows, we will use several times without explicit mention the Young,
H\"older and Poincar\'e inequalities.
We will also perform several formal computations, justified within a suitable
regularization scheme.

We introduce the (nonnegative) functional
$$
\Pi(t)=-\frac12 \int_0^\infty\mu'(s)\|\eta^t(s)\|^2_{\M}\d s,
$$
satisfying the equality (see e.g.\ \cite{Terreni})
$$
\Pi=-\l T \eta, \eta \r_{\M}.
$$
Hence, recalling that $b(u,u,u)=0$,
the basic multiplication of \eqref{PROBLEM} by $U$ in $\H$ gives
\begin{equation}
\label{diffeq1}
\frac{\d }{\d t} \E + \beta\|u\|^2_{-\vartheta} + \Pi = \l f, u \r.
\end{equation}

Next, in order to handle the possible singularity of $\mu$ at zero,
borrowing an idea from~\cite{PAT}
we fix $s_\ast>0$ such that
$$\int_{0}^{s_\ast}\mu(s)\d s \leq \frac{\kappa}{2}.$$
Setting
$$
\mu_\ast(s)=
\begin{cases}
\mu(s_\ast), & 0<s\leq s_\ast,\\
\mu(s), & s>s_\ast,
\end{cases}
$$
we consider the further functional
$$
\Phi(t)=-\frac{4}{\kappa}\int_0^\infty\mu_\ast(s)\l \eta^t(s),u(t)\r_1 \d s.
$$

\begin{lemma}
\label{PHIZERO}
For any $\varepsilon>0$, we have the differential inequality
\begin{equation}
\label{diffeq2}
\frac{\d}{\d t}\Phi+ \|u\|^2_1\leq \frac{8\mu(s_\ast)}{\kappa^2}\Pi
+ \frac{4}{\alpha\kappa\varepsilon} \| \eta\|_\M^2+ \alpha \varepsilon\| \partial_t u\|_1^2.
\end{equation}
\end{lemma}

\begin{proof}
We start from the identity
$$
\frac{\d}{\d t}\Phi+\frac4\kappa \int_0^\infty\mu_\ast(s)\|u\|^2_1 \d s=
-\frac4\kappa \int_0^\infty\mu_\ast(s)\l T \eta(s), u\r_1 \d s
 -\frac4\kappa \int_0^\infty\mu_\ast(s) \l \eta(s),\partial_ t u \r_1 \d s.
$$
According to the assumptions on $\mu$ and the definition of $\mu_\ast$, we get
$$
\frac4\kappa \int_{0}^\infty\mu_\ast(s)\|u\|^2_1 \d s
\geq \frac4\kappa \int_{s_\ast}^\infty\mu_\ast(s)\|u\|^2_1 \d s\geq 2 \|u\|^2_1.
$$
Integrating by parts, and using \eqref{dafermos}, we find the controls
\begin{align*}
-\frac4\kappa \int_0^\infty\mu_\ast(s)\l T \eta(s), u\r_1 \d s  &
= -\frac4\kappa \int_{s_\ast}^{\infty} \mu'(s) \l\eta(s),u\r_1 \d s \\
& \leq \frac{4\sqrt{2\mu(s_\ast)}}{\kappa} \| u\|_1 \Pi^{\frac12} \\
&\leq \| u\|_1^2 + \frac{8\mu(s_\ast)}{\kappa^2}\Pi,
\end{align*}
and
\begin{align*}
-\frac4\kappa \int_0^\infty\mu_\ast(s)\l \eta(s),\partial_t u \r_1\d s
&\leq \frac4\kappa \int_0^\infty\mu_\ast(s) \| \eta(s)\|_1\| \partial_t u\|_1 \d s \\
& \leq  \frac4\kappa \int_0^\infty\mu(s) \| \eta(s)\|_1\| \partial_t u\|_1 \d s\\
& \leq \alpha \varepsilon\| \partial_t u\|_1^2+
 \frac{4}{\alpha\kappa\varepsilon} \| \eta\|_\M^2.
\end{align*}
Collecting the estimates above, we meet the thesis.
\end{proof}

Finally, we define the functional
$$
\Psi(t)=2\beta \|u(t)\|^2_{-\vartheta}.
$$

\begin{lemma}
\label{PSIZERO}
We have the differential inequality
\begin{equation}
\label{diffeq3}
\frac{\d }{\d t}\Psi  +2\alpha\|\partial_t u\|^2_1
\leq C \|\eta\|_{\M}^2+ C \|f\|^2+C\|u\|\|u\|_1^{3}.
\end{equation}
\end{lemma}

\begin{proof}
Differentiating in time we find
$$
\frac{\d }{\d t}\Psi +4\|\partial_t u\|^2 + 4\alpha \|\partial_t u\|^2_1
 = 4\l f, \partial_t u\r-4\l \eta, \partial_t u\r_{\M} -4b(u,u,\partial_t u).
$$
Since
$$4\l f, \partial_t u\r-4\l \eta, \partial_t u\r_{\M}\leq \alpha\|\partial_t u\|_1^2
+ C \|\eta\|_{\M}^2+ C \|f\|^2,
$$
while, by interpolation and the Young inequality,
$$
-4b(u,u,\partial_t u)\leq
C\|u\|^{\frac12}\|u\|_1^{\frac32 } \|\partial_t u\|_1
\leq \alpha\|\partial_t u\|_1^2 +C\|u\|\|u\|_1^{3},
$$
we obtain \eqref{diffeq3}.
\end{proof}

\subsection{Proof of Theorem \ref{Th-EE}}
Denoting
$$
\nu=\min\Big\lbrace \frac{ \alpha \kappa \delta}{32},1\Big\rbrace,
$$
we introduce the family of energy functionals depending on the parameter $\varepsilon>0$
$$
\Lambda_\varepsilon(t)=\E(t)+\nu \varepsilon\Phi(t)+ \varepsilon^2\Psi(t).
$$
It is apparent that the control
\begin{equation}
\label{nore}
\frac12\E(t)\leq \Lambda_\varepsilon(t)\leq 2\E(t)
\end{equation}
holds for any $\varepsilon$ small enough.
Collecting \eqref{diffeq1}-\eqref{diffeq3} we deduce
the family of differential inequalities
$$
\ddt\Lambda_\varepsilon
+\nu\varepsilon\|u\|_1^2
+\Big(1- \frac{8\mu(s_\ast) \nu \varepsilon}{\kappa^2} \Big)\Pi
-\Big( \frac{4\nu}{\alpha \kappa}+ C\varepsilon^2\Big) \|\eta\|_{\M}^2
+\alpha \varepsilon^2\|\partial_t u\|_1^2
\leq \Theta_\varepsilon,
$$
having set
$$
\Theta_\varepsilon=-\beta\| u\|_{-\vartheta}^2
+\l f,u\r+C\|f\|^2+C \varepsilon^2 \|u\|\|u\|_1^{3}.
$$
Since by \eqref{dafermos}
$$
\frac{\delta}{2}\|\eta\|^2_{\M}\leq \Pi,
$$ from the very definition of $\nu$ we find
\begin{align*}
\Big(1- \frac{8\mu(s_\ast) \nu \varepsilon}{\kappa^2} \Big)\Pi
-\Big( \frac{4\nu}{\alpha \kappa}+ C\varepsilon^2\Big) \|\eta\|_{\M}^2
&\geq \frac12 \Pi-\Big( \frac{4\nu}{\alpha \kappa}
+ C\varepsilon^2\Big) \|\eta\|_{\M}^2\\
&\geq \Big( \frac{\delta}{4}- \frac{4\nu}{\alpha \kappa}
-C\varepsilon^2\Big) \|\eta\|_{\M}^2 \\
&\geq \frac{\delta}{16} \|\eta\|_{\M}^2,
\end{align*}
provided that $\varepsilon>0$ is sufficiently small.
This gives
$$
\ddt\Lambda_\varepsilon+ \nu \varepsilon\| u\|_1^2
+\frac{\delta}{16}\| \eta\|_{\M}^2+\alpha \varepsilon^2\|\partial_t u\|_1^2
\leq \Theta_\varepsilon.
$$
In light of~\eqref{Poincare} and~\eqref{nore},
it is apparent that the latter inequality
can be rewritten in the form
$$
\ddt\Lambda_\varepsilon+2\varkappa\varepsilon \Lambda_\varepsilon
+\alpha \varepsilon^2\|\partial_t u\|_1^2
\leq \Theta_\varepsilon,
$$
for every $\varepsilon>0$ small and a suitably fixed $\varkappa>0$.
We are left to estimate the term $\Theta_\varepsilon$.
To this end, let us put
$$p=\frac{4(1+\vartheta)}{1+2\vartheta},\qquad
q=\frac{3+4\vartheta}{1+2\vartheta},\qquad
r=\frac{\vartheta-\varrho}{1+\vartheta}.
$$
Observe that $p,q,r$ comply with the hypotheses of Lemma~\ref{LEMMAPatonwall}.
In particular, $r<1$.
Then, by ~\eqref{nore} and Lemma~\ref{INTERPOLINO} with $a=-\vartheta$, $b=-\varrho$, $c=1$
and $\varpi=r$,
\begin{align*}
\langle f,u\rangle&
\leq \|f\|_{\varrho} \|u\|_1^{r}\|u\|_{-\vartheta}^{1-r}\\
\noalign{\vskip1.5mm}
&\leq \frac\beta2\|u\|_{-\vartheta}^2+ C \|f\|_\varrho^{\frac{2}{1+r}} \| u\|_1^{\frac{2r}{1+r}}\\
&\leq \frac\beta2\|u\|_{-\vartheta}^2 + \varkappa\varepsilon \Lambda_\varepsilon+ \frac{C\| f\|_\varrho^2}{\varepsilon^r}
\end{align*}
By the same token, using now
$a=-\vartheta$, $b=0$, $c=1$
and $\varpi=\vartheta/(1+\vartheta)$,
and subsequently applying the Young inequality with exponents $2(1+\vartheta)$ and $p/2$,
\begin{align*}
C \varepsilon^2 \|u\|\|u\|_1^3
&\leq C \varepsilon^2 \|u\|_{-\vartheta}^{\frac1{1+\vartheta}}
\|u\|_1^{\frac{3+4\vartheta}{1+\vartheta}}\\
&\leq C \varepsilon^2 \|u\|_{-\vartheta}^{\frac1{1+\vartheta}}
\Lambda_\varepsilon^{\frac{2q}{p}}\\
&\leq \frac{\beta}{2} \|u\|_{-\vartheta}^2
+ C \varepsilon^{p}\Lambda_\varepsilon^q.
\end{align*}
Summarizing (recall that $V^\varrho\subset\HH$), we obtain the estimate
$$
\Theta_\varepsilon\leq \varkappa\varepsilon \Lambda_\varepsilon
+C \varepsilon^{p}\Lambda_\varepsilon^q+\frac{C \| f\|_\varrho^2}{\varepsilon^r},
$$
and we end up with the differential inequality
\begin{equation}
\label{lastineq}
\ddt\Lambda_\varepsilon+\varkappa\varepsilon \Lambda_\varepsilon
+\alpha \varepsilon^2\|\partial_t u\|_1^2
\leq C \varepsilon^{p}\Lambda_\varepsilon^q+\frac{C \| f\|_\varrho^2}{\varepsilon^r}.
\end{equation}
Using again \eqref{nore}, the desired conclusion follows from an application
of Lemma~\ref{LEMMAPatonwall}.
\qed

\medskip
Once one has Theorem~\ref{Th-EE}, an integration of \eqref{lastineq}
for a fixed $\varepsilon>0$
provides a useful estimate needed in the sequel.

\begin{corollary}
\label{ptu1}
Let the assumptions of Theorem \ref{Th-EE} hold. Then for every $t\geq 0$
\begin{equation}
\int_0^{t} \| \partial_t u(y)\|_1^2\d y\leq (1+t)\Q(\E(0)).
\end{equation}
\end{corollary}

\section{Exponential Decay of Solutions}
\label{secEDS}

\noindent
In absence of a forcing term,
Theorem \ref{Th-EE} establishes the exponential decay of the energy.

\begin{corollary}
\label{corTh-EE}
Let $\beta>0$ be fixed, and let $f\equiv 0$. Then, for any initial datum $U_0\in\H$,
the corresponding energy fulfills the decay estimate
$$
\E(t)\leq \Q(\E(0))\e^{-\omega t},
$$
where $\omega>0$ is a universal constant depending only on the structural parameters of the problem.
\end{corollary}

Nonetheless, the conclusion above remains true also when $\beta=0$, that is, when no extra damping
of the form $\beta A^{-\vartheta} u$ is present. The result, however, is proved in a different way,
with an argument that cannot be exported to the case of a nonzero $f$.

\begin{theorem}
\label{expstab}
Let $\beta=0$, and let $f\equiv 0$. Then, for any initial datum $U_0\in\H$,
the corresponding energy fulfills the decay estimate
$$
\E(t)\leq \Q(\E(0))\e^{-\omega t},
$$
where $\omega>0$ is a universal constant depending only on the structural parameters of the problem.
\end{theorem}

\begin{proof}
For an arbitrarily given $R\geq0$, let us consider an initial datum
$U_0\in \B_{\H}(R)$. Since $f\equiv 0$, we readily see from~\eqref{diffeq1}
that
\begin{equation}
\label{stabcontrol}
\E(t)\leq \E(0).
\end{equation}
Setting
$$
\nu= \min\Big\lbrace \frac{\alpha\kappa\delta}{32}, 1 \Big\rbrace,
$$
for $\varepsilon>0$ small to be fixed later, we define this time
$$
\Lambda(t)= \E(t)+ \nu\varepsilon \Phi(t),
$$
with $\Phi$ as in the previous section.
It is apparent that
$$
\frac12 \E(t)\leq \Lambda(t)\leq 2\E(t).
$$
In addition, we have the differential inequality
$$
\ddt\Lambda+ \nu\varepsilon\|u\|_1^2
+\Big(1- \frac{8\mu(s_\ast) \nu \varepsilon}{\kappa^2} \Big)\Pi
-\frac{4\nu}{\alpha \kappa} \|\eta\|_{\M}^2
\leq \alpha\nu\varepsilon^2 \|\partial_t u\|_1^2.
$$
Repeating the calculations of the former proof, we arrive at the inequality
\begin{equation}
\label{stabineq1}
\ddt\Lambda +2\varkappa\varepsilon \Lambda\leq
 \alpha\varepsilon^2 \| \partial_t u\|_1^2,
\end{equation}
for some $\varkappa>0$.
In order to bound the right-hand side, we multiply
the first equation of~\eqref{PROBLEM} by
$2\varepsilon^2 \partial_t u$, to get
$$
2\alpha \varepsilon^2 \|\partial_t u\|^2_1=-2\varepsilon^2 \|\partial_t u\|^2
-2\varepsilon^2\l \eta, \partial_t u\r_{\M}
-2\varepsilon^2 b(u,u,\partial_t u).
$$
Invoking the Poincar\'e inequality~\eqref{Poincare}, standard computations
together with~\eqref{stabcontrol} yield
$$\alpha \varepsilon^2 \| \partial_t u\|_1^2 \leq
C \varepsilon^2 \| u\|_1^4+C \varepsilon^2\|\eta\|^2_\M\leq C(\E(0)+1)\varepsilon^2\Lambda.
$$
Accordingly, we end up with
$$\ddt\Lambda +\big[2\varkappa-C(\E(0)+1)\varepsilon\big]\varepsilon\Lambda\leq 0.
$$
At this point, we choose $\varepsilon$ small enough such that $C(\E(0)+1)\varepsilon\leq\varkappa$,
and an application of the Gronwall lemma entails
\begin{equation}
\label{stab}
\E(t)\leq \Q(\E(0))\e^{-\omega t},
\end{equation}
with $\omega=\varkappa\varepsilon$.
This is the desired inequality, except that the exponential rate $\omega$ depends on $\E(0)$.
To complete the argument, we use a rather standard trick
of semigroup theory. Due to~\eqref{stab}, there is a time $t_0\geq 0$, depending on $\E(0)$, such that
$\E(t_0)\leq 1$.
Hence, for $t\geq t_0$ we can repeat the argument above, obtaining
\begin{equation}
\label{stabcontrol1}
\E(t)\leq \Q(1)\e^{-\omega(t-t_0)}, \quad \forall\, t\geq t_0,
\end{equation}
where now $\omega>0$ is independent of $\E(0)$.
Collecting \eqref{stab}, where $\omega=\omega(\E(0))$, and \eqref{stabcontrol1},
we reach the desired conclusion.
\end{proof}

\section{Regular Exponentially Attracting Sets}
\label{secREAS}

\noindent
It is well known that dynamical systems generated by equation with memory do not regularize in finite time,
due to the intrinsic hyperbolicity of the memory component. In particular, this prevents the existence
of absorbing sets having higher regularity than the initial data. Nonetheless, one can still hope that
trajectories originating from bounded sets are exponentially attracted by more regular bounded sets.

\begin{definition}
A bounded set $\mathcal{B}_\star$ is said to be {\it exponentially
attracting} for $S(t)$ in $\H$ if there exists
$\omega>0$ such that
$$
\text{dist}_{\H}(S(t)\mathcal{B},\mathcal{B}_\star)\leq \Q(\| \mathcal{B}\|_{\H})\e^{-\omega t}
$$
for every bounded subset $\mathcal{B}\subset \H$.
\end{definition}

Here, with standard notation,
$$\text{dist}_{\H}(\mathcal{B}_1,\mathcal{B}_2)=\sup_{b_1\in\mathcal{B}_1}\inf_{b_2\in \mathcal{B}_2}
\|b_1-b_2\|_\H$$
is the Hausdorff semidistance in $\H$ between two (nonempty)
sets $\mathcal{B}_1$ and $\mathcal{B}_2$.

\begin{proposition}
\label{REAS}
There exists $R_\star>0$ such that the ball
$$\mathcal{B}_\star=\B_{\H^1}(R_\star)$$
is exponentially attracting for $S(t)$.
\end{proposition}

It is enough showing that the ball $\mathcal{B}_\star$ exponentially attracts
the absorbing set $\mathcal{B}_0$ found in Section~\ref{secD}.
To this end, for every initial data
$U_0=(u_0,\eta_0) \in \mathcal{B}_0$, in the same spirit of \cite[Chapter I]{Temam}
we decompose the solution
$S(t)U_0$ into the sum
$$
S(t)U_0= L(t)U_0+ K(t)U_0,
$$
where the maps
$$L(t)U_0=(v(t),\xi^t)\qquad\text{and}\qquad
K(t)U_0=(w(t),\zeta^t)$$
solve the problems
\begin{equation}
\label{LPROBLEM}
\begin{cases}
\partial_t  \left(v + \alpha Av\right) + B(u,v) + \displaystyle
\int_0^\infty  \mu(s) A \xi(s)\d s =0, & \\
\partial_t \xi = T \xi + v,&\\
\noalign{\vskip2mm}
(v(0),\xi^0)=(u_0,\eta_0),
\end{cases}
\end{equation}
and
\begin{equation}
\label{KPROBLEM}
\begin{cases}
\partial_t  \left(w + \alpha Aw\right) + B(u,w) + \displaystyle
\int_0^\infty  \mu(s) A \zeta(s)\d s + \beta A^{-\vartheta} w =\tilde{f}, & \\
\partial_t \zeta = T \zeta + w,&\\
\noalign{\vskip2mm}
(w(0),\zeta^0)=(0,0),
\end{cases}
\end{equation}
with
$$
\tilde{f}(t)=f-\beta A^{-\vartheta} v(t).
$$
Existence and uniqueness of these problems
are rather standard issues, since \eqref{LPROBLEM} and \eqref{KPROBLEM}
are linear with a globally Lipschitz perturbation.
Following a standard procedure, we will show that system \eqref{LPROBLEM}
is exponentially stable, whereas the solutions to \eqref{KPROBLEM} are
uniformly bounded in the more regular space $\H_1$.

\smallskip
In what follows, the generic positive constant $C$ may depend on
$\|f\|$ as well as on the radius $R_0$ of the absorbing set $\mathcal{B}_0$.
In particular, by Theorem~\ref{Th-EE} we know that
\begin{equation}
\label{estS}
\| S(t)U_0\|_{\H}\leq C.
\end{equation}
The proof of Proposition \ref{REAS} is an immediate consequence
of the following two lemmas.

\begin{lemma}
\label{Lest}
For every initial datum $U_0\in \mathcal{B}_0$,
we have the estimate
$$
\|L(t)U_0\|_{\H}\leq C \e^{-\omega t},
$$
where $\omega>0$ is a universal constant.
\end{lemma}

\begin{proof}
Let $\Lambda$ be as in the proof of
Theorem~\ref{expstab}, with $(v,\xi)$ in place of $(u,\eta)$. It is apparent that
$$
\frac14 \| L(t)U_0\|_{\H}^2 \leq \Lambda(t)\leq \| L(t)U_0\|_{\H}^2.
$$
Repeating the same reasonings, making use of~\eqref{estS}
to handle the term $b(u,v,\partial_t v)$, we find
the differential inequality
$$
\ddt\Lambda +2\omega \Lambda\leq 0
$$
for some $\omega>0$, and the Gronwall lemma completes the argument.
\end{proof}

\begin{lemma}
\label{Klemma}
For every initial datum $U_0\in \mathcal{B}_0$,
we have the estimate
$$
\| K(t)U_0\|_{\H^1}\leq C.
$$
\end{lemma}

\begin{proof}
We first note that, by virtue of \eqref{estS} and the previous Lemma~\ref{Lest},
\begin{equation}
\label{estK0}
\|K(t)U_0\|_{\H}\leq C
\end{equation}
and
$$
\|\tilde{f}(t)\|\leq C.
$$
We introduce the analogous functionals of Section~\ref{secD} for the variable
$(w,\zeta)$ in higher
order spaces. Namely,
$$
\E_1(t)= \frac12 \|K(t)U_0\|_{\H^1}^2,
$$
and
$$
\Pi_1(t)= -\frac12 \int_0^\infty\mu'(s)\|\zeta^t(s)\|^2_{\M^1}\d s.
$$
In particular, we have the identity
$$\Pi_1=-\l T A^{\frac12}\zeta, A^{\frac12}\zeta \r_{\M}
=-\l T \zeta,\zeta \r_{\M^1}.$$
Hence, testing the first equation of \eqref{KPROBLEM} by $Aw$ in $\HH$
and the second one by $\zeta$ in $\M^1$, we obtain
\begin{equation}
\label{Wdiffeq1}
\frac{\d}{\d t} \E_1 +b(u,w,Aw)+ \beta \|w\|_{1-\vartheta}^2+
\Pi_1=(\tilde{f},Aw).
\end{equation}
In a similar fashion, we define
\begin{align*}
\Phi_1(t) &=-\frac{6}{\kappa}\int_0^\infty\mu_\ast(s)\l \zeta^t(s),w(t)\r_2 \d s,\\
\noalign{\vskip1mm}
\Psi_1(t) &=\beta \|w(t)\|^2_{1-\vartheta}.
\end{align*}
Arguing as in Lemma~\ref{PHIZERO}, we deduce the differential inequality
\begin{equation}
\label{Wdiffeq2}
\frac{\d}{\d t}\Phi_1+2\|w\|^2_2\leq \frac{18\mu(s_\ast)}{\kappa^2}
\Pi_1+ \frac{9}{\alpha\kappa\varepsilon} \|\zeta\|_{\M^1}^2
+\alpha \varepsilon \| \partial_t w\|_2^2,
\end{equation}
while
$$
\frac{\d }{\d t}\Psi_1 +2\|\partial_t w\|_1^2
+ 2\alpha \|\partial_t w\|_2^2=  2\l \tilde{f}, A \partial_t w\r
-2\l \zeta, \partial_t w\r_{\M^1}-2b(u,w,A\partial_t w).
$$
Exploiting \eqref{estS},
$$
-2b(u,w,A \partial_t w)\leq C \|w\|_2\|\partial_t w\|_2,
$$
and we readily get
\begin{equation}
\label{Wdiffeq3}
\frac{\d }{\d t}\Psi_1+ \alpha \|\partial_t w\|_2^2
\leq  C \|\zeta\|_{\M^1}^2+ C \| w\|_2^2+C.
\end{equation}
At this point, setting
$$
\nu=\min\Big\lbrace \frac{\alpha \kappa \delta}{72}, 1\Big\rbrace,
$$
for some $\varepsilon>0$ small to be determined later, we define the functional
$$
\Lambda_1 (t)= \E_1(t)+\nu\varepsilon\Phi_1(t)+ \varepsilon^2 \Psi_1(t),
$$
which fulfills the controls
$$
\frac12 \E_1(t) \leq \Lambda_1(t)\leq 2\E_1(t).
$$
Collecting \eqref{Wdiffeq1}-\eqref{Wdiffeq3}, we find the differential inequality
\begin{align*}
&\frac{\d}{\d t} \Lambda_1+\nu\varepsilon
\| w\|_2^2+ \Big(1- \frac{18\mu(s_\ast) \nu \varepsilon}{\kappa^2} \Big)
\Pi_1-\Big(\frac{9\nu}{\alpha\kappa} + C \varepsilon^2\Big)
\| \zeta\|_{\M^1}^2\leq \Theta_1,
\end{align*}
where
$$
\Theta_1=-\nu \varepsilon\|w\|_2^2+C\varepsilon^2\|w\|_2^2+(\tilde{f},Aw)
-b(u,w,Aw)+ C.
$$
On account of \eqref{estS}-\eqref{estK0},
\begin{align*}
(\tilde{f},Aw)-b(u,w,Aw) &\leq \| \tilde{f}\|\| w\|_2+\|u\|_{L^6} \| \nabla w\|_{L^3} \| Aw\|\\
&\leq \frac{\nu\varepsilon}{4} \| w\|_2^2+ C \|w\|_1^{\frac12}\|w \|_2^{\frac32}+C\\
&\leq \frac{\nu\varepsilon}{2} \| w\|_2^2+ C.
\end{align*}
It is then apparent that (for $\varepsilon>0$ small)
$$\Theta_1\leq C.$$
Note that the constants $C$ above depend on $\varepsilon$, which however will be eventually fixed.
Indeed, once $\varepsilon$ is chosen suitably small, and recasting almost word by word
the proof of Theorem~\ref{Th-EE}, we end up with
$$
\frac{\d}{\d t} \Lambda_1+ \omega \Lambda_1 \leq C,
$$
for some $\omega>0$.
Since $\Lambda(0)=0$, a final application of the Gronwall lemma will do.
\end{proof}

For a semigroup $S(t)$ satisfying the continuity property
$S(t)\in{\mathcal C}(\H,\H)$ for every $t\geq 0$, as in our case,
having a \textit{compact} exponentially attracting set
is a sufficient condition in order for the \textit{global attractor} to exist (see e.g.\ \cite{Temam}).
By definition, this is the (unique) compact set $\mathfrak{A}\subset \H$ which is at the same time
\begin{itemize}
\item[$\diamond$] fully invariant, i.e.\
$S(t)\mathfrak{A}=\mathfrak{A}$ for every $t\geq 0$;
\item[$\diamond$] attracting for the semigroup, i.e\
$$
\lim_{t\to\infty}\big[\text{dist}_{\H}(S(t)\mathcal{B},\mathfrak{A})\big]=0
$$
for every bounded subset $\mathcal{B}\subset \H$.
\end{itemize}
Unfortunately, our attracting set ${\mathcal B}_\star$, although closed and bounded in $\H^1$,
is not compact in $\H$. Indeed, even if the embedding
$\VV\subset\HH$ is compact, the same cannot be said for the embedding $\M^1\subset \M$
(see \cite{PataZucchi} for a counterexample to compactness). Accordingly, the embedding $\H^1\subset \H$
is in general not compact as well.
Nevertheless, there is a general argument devised in \cite{PataZucchi} that allows to recover
the sought compactness with a little effort, producing a compact set ${\mathcal B}_\star'\subset {\mathcal B}_\star$
which is still exponentially attracting. In turn, this entails the existence of the global attractor
$\mathfrak{A}$. We do not enter into more details, since in the next sections we will
prove the existence of an exponential attractor. As a byproduct, this will yield the existence of
$\mathfrak{A}$, along with the finiteness of its fractal dimension.

\section{Exponential Attractors}
\label{secEA}

\begin{definition}
A compact set $\mathfrak{E}\subset \H$ is an \textit{exponential attractor} for $S(t)$ if
\begin{itemize}
\item[$\diamond$] $\mathfrak{E}$ is positively invariant, i.e.\ $S(t)\mathfrak{E}\subset\mathfrak{E}$ for every $t\geq 0$;
\item[$\diamond$] $\mathfrak{E}$ is exponentially attracting for the semigroup;
\item[$\diamond$] $\mathfrak{E}$ has finite fractal dimension in $\H$.
\end{itemize}
\end{definition}

Recall that the fractal dimension of $\mathfrak{E}$ in $\H$ is defined as
$$\text{dim}_{\H}(\mathfrak{E})=\limsup_{\varepsilon\to 0}
\frac{\ln N(\varepsilon)}{\ln \frac1\varepsilon},$$
where $N(\varepsilon)$ is
the smallest number of $\varepsilon$-balls of $\H$ necessary to cover $\mathfrak{E}$.

\smallskip
The main result of the paper reads as follows.

\begin{theorem}
\label{EAVoigt}
The dynamical system $S(t)$ on $\H$ possesses an exponential attractor $\mathfrak{E}$,
which is bounded in $\H^1$.
\end{theorem}

As a consequence of the existence of a compact attracting set,
$S(t)$ possesses the global attractor $\mathfrak{A}$,
which is the smallest among the compact attracting sets (hence contained in the exponential attractor $\mathfrak{E}$).

\begin{corollary} \label{GAVoigt}
The dynamical system $S(t)$ on $\H$ possesses the global attractor $\mathfrak{A}$.
Moreover, $\mathfrak{A}$ has finite fractal dimension in $\H$ and is a bounded in $\H^1$.
\end{corollary}

The proof of Theorem \ref{EAVoigt}, carried out in the next section, is based on
an abstract result from \cite{DGP} (see Theorem 5.1 therein), that we report here below as a lemma,
in a version specifically tailored to fit our particular problem.
To this end, we will make use of the projections $\mathbb{P}_1$ and $\mathbb{P}_2$
of $\H$ onto its components $\VV$ and $\M$, namely,
$$
\mathbb{P}_1 (u,\eta) = u\qquad\text{and}\qquad  \mathbb{P}_2(u,\eta) = \eta.
$$

\begin{lemma}
\label{EA}
Let the following assumptions hold.
\smallskip
\begin{itemize}
\item[\textbf{(i)}] There exists $R_\star > 0$ such that the ball $\mathcal{B}_\star=\B_{\H^1} (R_\star)$
is exponentially attracting.
\smallskip
\item[\textbf{(ii)}] For every $R \geq 0$ and every $\theta > 0$
sufficiently large,
$$
\int_\theta^{2\theta}\| \partial_t u(t)\|_{1}^2 \d t\leq \Q(R+\theta),
$$
for all $u(t) = \mathbb{P}_1 S(t)U_0$ with $U_0 \in \B_{\H^1} (R)$.
\smallskip
\item[\textbf{(iii)}] There exists $R_1> 0$ with the following property:
for any given $R\geq 0$, there exists a nonnegative function $\psi$ vanishing at infinity
such that
$$
\|S(t)U_0\|_{\H^1}\leq \psi(t) + R_1,
$$
for all $U_0 \in \B_{\H^1}(R)$.
\smallskip
\item[\textbf{(iv)}] For every fixed $R\geq 0$, the semigroup $S(t)$ admits
a decomposition of the form
$$S(t)=L(t)+K(t)$$
satisfying for all initial data $U_{0i}\in \B_{\H^1}(R)$
\begin{align*}
\|L(t)U_{01}-L(t)U_{02}\|_{\H} &\leq \psi(t)\|U_{01}-U_{02}\|_{\H},\\
\|K(t)U_{01}-K(t)U_{02}\|_{\H^1} &\leq \Q(t)\|U_{01}-U_{02}\|_{\H}.
\end{align*}
Here, both $\Q$ and
the nonnegative function $\psi$ vanishing at infinity depend on $R$.
Moreover, the function
$$\bar\zeta^t = \mathbb{P}_2 K(t)U_{01}-\mathbb{P}_2 K(t)U_{02}$$
fulfills the Cauchy problem
$$
\begin{cases}
\partial_t\bar\zeta^t= T\bar\zeta^t+\bar w(t),\\
\bar\zeta^0=0,
\end{cases}
$$
for some $\bar w$ satisfying the estimate
$$
\|\bar w(t)\|_{1} \leq \Q(t)\|U_{01}-U_{02}\|_{\H}.
$$
\end{itemize}
Then $S(t)$ possesses an exponential attractor $\mathfrak{E}$ contained in the ball $\B_{\H^1}(R_1)$.
\end{lemma}

\begin{remark}
Actually, in the abstract result from \cite{DGP} a further assumption is needed,
involving a certain operator that in our case is just the identity (and the assumption is trivially satisfied).
\end{remark}

\section{Proof of Theorem \ref{EAVoigt}}
\label{secPT}

\noindent
The proof amounts to verifying the four points of the above Lemma~\ref{EA}.
Indeed, {\bf (i)} is the content of Proposition~\ref{REAS}, while {\bf (ii)}
is an immediate consequence of the continuous
embedding $\H^1\subset\H$ and Corollary~\ref{ptu1}. Accordingly, we are left to show the
validity of {\bf (iii)} and {\bf (iv)}. In what follows,
the generic positive constant $C$ may depend on $\|f\|$ and on the radius $R_0$ of the absorbing set $\mathcal{B}_0$.

\subsection*{$\bullet$ Verifying (iii)}
Given $R\geq 0$, let us consider the ball $\B_{\H^1}(R)$. We easily infer from
the continuous embedding $\H^1\subset \H$ that
$$
\B_{\H^1}(R)\subset \B_{\H}(\Q(R)).
$$
Therefore, on account of Theorem \ref{Th-EE}, there exists $t_\e=t_\e(R)$ such that
\begin{equation}
\label{abslowernorm}
\begin{cases}
\|S(t)\B_{\H^1}(R)\|\leq \Q(R), & \forall \, t\leq t_\e,\\
\| S(t)\B_{\H^1}(R)\|\leq R_0, & \forall \, t\geq t_\e.
\end{cases}
\end{equation}
Taking an arbitrary $U_0 \in \B_{\H^1}(R)$, we define the higher-order
energy functional
$$
\E_1(t)= \frac12 \|S(t)U_0\|_{\H^1}^2,
$$
and the nonnegative functional
$$
\Pi_1(t)= -\frac12 \int_0^\infty\mu'(s)\|\eta^t(s)\|^2_{\M^1}\d s.
$$
We recall the identity
$$\Pi_1=-\l T \eta,\eta \r_{\M^1}.$$
Testing the first equation of \eqref{PROBLEM} by $Au$ in $\HH$ and the
 second one by $\eta$ in $\M^1$, we obtain the differential equality
\begin{equation}
\label{Udiffeq1}
\frac{\d}{\d t} \E_1 + \beta \|u\|_{1-\vartheta}^2+ \Pi_1(\eta)= (f,Au)-b(u,u,Au).
\end{equation}
For any $t\leq t_\e$, according to \eqref{abslowernorm} and exploiting the standard Sobolev embeddings, we have
\begin{align*}
(f,Au)-b(u,u,Au)&\leq \|f\|\|u\|_2 +\| u\|_{L^6} \|\nabla u\|_{L^3}\| u\|_2\\
&\leq C\|u\|_2 + C \| u\|_1 \| u\|_2^2\\
&\leq \Q(R)\|u\|_2^2+C,
\end{align*}
which in turn gives
$$
\frac{\d}{\d t} \E_1 \leq \Q(R)\E_1+ C.
$$
The Gronwall lemma entails
\begin{equation}
\label{highest1}
\E_1(t)\leq \Q(R)\e^{\Q(R)t_\e}=\Q(R), \quad \forall \, t\leq t_\e.
\end{equation}
In order to show the existence of an absorbing set for the semigroup $S(t)$
on $\H^1$, and similarly to the proof of Lemma \ref{Klemma}, we
define some further functionals.
For any $t\geq t_\e$, let
$$
\Phi_1(t)=-\frac{6}{\kappa}\int_0^\infty\mu_\ast(s)\l \eta^t(s),u(t)\r_2 \d s,
$$
and
$$
\Psi_1(t)=\beta \|u(t)\|^2_{1-\vartheta}.
$$
Arguing as in Lemma \ref{Klemma},
\begin{equation}
\label{Udiffeq2}
\frac{\d}{\d t}\Phi_1+ 2\|u\|^2_2
\leq \frac{18\mu(s_\ast)}{\kappa^2}\Pi_1
+\frac{9}{\alpha\kappa\varepsilon} \| \eta\|_{\M^1}^2
+\alpha \varepsilon \| \partial_t u\|_2^2,
\end{equation}
whereas a differentiation in time yields
$$
\frac{\d }{\d t}\Psi_1 +2\|\partial_t u\|_1^2
+2\alpha \|\partial_t u\|_2^2
= 2\l f, A\partial_t u\r -2\l \eta,\partial_t u\r_{\M^1}
-2b(u,u,A\partial_t u).
$$
Making use of \eqref{abslowernorm},
$$
-2b(u,u,A\partial_t u)\leq C \|u\|_2\|\partial_t u\|_2,
$$
and we end up with
\begin{equation}
\label{Udiffeq3}
\frac{\d }{\d t}\Psi_1 +\alpha \|\partial_t u\|_2^2
\leq  C\| u\|_2^2 + C \|\eta\|_{\M^1}^2+C.
\end{equation}
Denoting
$$
\nu=\min\Big\lbrace \frac{\alpha \kappa \delta}{72}, 1\Big\rbrace,
$$
we define
$$
\Lambda_1(t)=\E_1(t)+\nu\varepsilon\Phi_1(t)+ \varepsilon^2 \Psi_1(t),
$$
depending on $\varepsilon>0$ to be determined later.
As customary, we have the controls
\begin{equation}
\label{nore4}
\frac12 \E_1(t) \leq \Lambda_1(t)\leq 2\E_1(t),
\end{equation}
provided $\varepsilon$ is small enough.
Moreover, adding \eqref{Udiffeq1}, \eqref{Udiffeq2} and
\eqref{Udiffeq3}, we are led to the differential inequality
$$
\frac{\d}{\d t} \Lambda_1+\nu\varepsilon\| u\|_2^2
+\Big(1-\frac{18\mu(s_\ast)\nu\varepsilon}{\kappa^2}\Big) \Pi_1
-\Big(\frac{9\nu}{\alpha\kappa}+C\varepsilon^2\Big)\|\eta\|_{\M^1}^2\leq \Theta_1,
$$
having set
$$
\Theta_1=-\nu\varepsilon\|u\|_2^2+(f,Au)-b(u,u,Au)+ C\varepsilon^2\|u\|^2_2+C.
$$
Since
$$
\frac{\delta}{2}\|\eta\|_{\M^1}^2\leq \Pi_1,
$$
we deduce that
\begin{align*}
\Big(1-\frac{18\mu(s_\ast)\nu\varepsilon}{\kappa^2}\Big) \Pi_1
-\Big(\frac{9\nu}{\alpha\kappa}+C\varepsilon^2\Big)\|\eta\|_{\M^1}^2&\geq
\frac12\Pi_1-\Big(\frac{9\nu}{\alpha\kappa}+C\varepsilon^2\Big)\|\eta\|_{\M^1}^2 \\
&\geq\Big(\frac\delta4- \frac\delta8-C\varepsilon^2\Big)\|\eta\|_{\M^1}^2\\
&\geq \frac{\delta}{16}\|\eta\|_{\M^1}^2,
\end{align*}
provided that $0<\varepsilon<\frac{\sqrt{\delta}}{4\sqrt{C}}$.
Accordingly, for $\varepsilon>0$ sufficiently small, 
$$
\ddt \Lambda_1+ \varkappa \varepsilon \Lambda_1\leq \Theta_1,
$$
for some $\varkappa>0$.
Regarding the right-hand side $\Theta_1$, by interpolation and \eqref{abslowernorm},
we have the estimate
\begin{align*}
(f,Au)-b(u,u,Au)& \leq \|f\|\|u\|_2 + \|u\|_{L^6}\|\nabla u\|_{L^3}\|u\|_2 \\
&\leq C\|u\|_2 + C\| u\|_2^{\frac32}\\
\noalign{\vskip1mm}
&\leq \varepsilon^2 \|u\|_2^2+ C.
\end{align*}
Then, up to choosing $\varepsilon>0$ small enough,
we find
$$
\Theta\leq C,
$$
which in turn entails
$$
\frac{\d}{\d t} \Lambda_1+ \varkappa \varepsilon \Lambda_1 \leq C.
$$
The Gronwall lemma on the time-interval $(t_\e,t)$ together with \eqref{nore4}
yield
\begin{equation}
\label{highest2}
\E_1(t)\leq \Q(R)\e^{-\omega (t-t_\e)}+ C,  \quad \forall \, t\geq t_\e,
\end{equation}
with $\omega=\varkappa\varepsilon$.
Collecting \eqref{highest1} and \eqref{highest2}, we readily get
$$
\E_1(t)\leq \Q(R)\e^{-\omega t}+ C,\quad \forall \, t\geq 0,
$$
implying the desired conclusion.
\qed

\subsection*{$\bullet$ Verifying (iv)}
Given any initial datum $U_0=(u_0,\eta_0)$,
we consider this time the trivial splitting
$$
S(t)U_0=L(t)U_0+K(t)U_0
$$
where the maps
$$L(t)U_0=(v(t),\xi^t)\qquad\text{and}\qquad
K(t)U_0=(w(t),\zeta^t)$$
solve the problems
$$
\begin{cases}
\partial_t  \left(v + \alpha Av\right)
+\displaystyle \int_0^\infty  \mu(s) A \xi(s)\d s=0, & \\
\partial_t \xi = T \xi + v,&\\
\noalign{\vskip2mm}
(v(0),\xi^0)=(u_0,\eta_0),
\end{cases}
$$
and
$$
\begin{cases}
\partial_t  \left(w + \alpha Aw\right)
+\displaystyle \int_0^\infty  \mu(s) A \zeta(s)\d s
=f-B(u,u)-\beta A^{-\vartheta}u, & \\
\partial_t \zeta = T \zeta + w,&\\
\noalign{\vskip2mm}
(w(0),\zeta^0)=(0,0).
\end{cases}
$$
Note that $L(t)$ is a strongly continuous linear semigroup on $\H$.
Besides, $L(t)$ is exponentially stable. This
can be easily seen by recasting the proof of Theorem~\ref{expstab}.

Let now $R\geq 0$ be fixed, and let
$U_{01},U_{02} \in \B_{\H^1}(R)$.
Along this proof, the generic positive constant $C$ is allowed to depend on $R$.
Then, we decompose the difference
$$
(\bar u(t),\bar\eta^t)=S(t)U_{01}-S(t)U_{02}$$
into the sum
$$
(\bar u(t),\bar\eta^t)=(\bar v(t),\bar\xi^t)+(\bar w(t),\bar\zeta^t),
$$
where
$$
(\bar v(t),\bar\xi^t)=L(t)U_{01}-L(t)U_{02}, \quad \text{and}
 \quad (\bar w(t),\bar \zeta^t)= K(t)U_{01}-K(t)U_{02}.
$$
We first note that, on account of $\textbf{(iii)}$,
\begin{equation}
\label{bound12}
\| S(t)U_{0i}\|_{\H^1}\leq C.
\end{equation}
Besides, the exponential stability of $L(t)$ implies the existence of a universal constant $\omega>0$
such that
\begin{equation}
\label{boundL12}
\|L(t)U_{01}-L(t)U_{02}\|_{\H}\leq C\e^{-\omega t} \|U_{01}-U_{02}\|_{\H}.
\end{equation}
We are left to prove the desired estimate for the difference
$(\bar w,\bar \zeta)$, solution to the system
$$
\begin{cases}
\partial_t  \left(\bar w + \alpha A\bar w\right)
+\displaystyle \int_0^\infty  \mu(s) A \bar\zeta(s)\d s
 =-B(\bar u,u_1)-B(u_2,\bar u)-\beta A^{-\vartheta}\bar u, & \\
\partial_t \bar\zeta = T \bar\zeta + \bar w,&\\
\noalign{\vskip2mm}
(\bar w(0),\bar\zeta^0)=(0,0).
\end{cases}
$$
To this end, introducing the higher-order energy
$$
\E_1(t)=\frac12 \|K(t)U_{01}-K(t)U_{02}\|_{\H^1}^2,
$$
we find the identity
$$
\ddt \E_1 -\frac12 \int_0^\infty\mu'(s)\|\bar\zeta^t(s)\|^2_{\M^1}\d s
=-b(\bar u,u_1,A\bar w)-b(u_2,\bar u,A\bar w)-\beta\l A^{-\vartheta}\bar u, A\bar w\r.
$$
Owing to \eqref{bound12}, and appealing to the embedding $\WW\subset [L^\infty(\Omega)]^3$, we have the controls
\begin{align*}
-b(\bar u,u_1,A\bar w)&\leq C \| u_1\|_2 \|\bar u\|_1 \|\bar w\|_2\\
&\leq C \|\bar v\|_1 \|\bar w\|_2+ C \|\bar w\|_2^2\\
&\leq C\|\bar w\|_2^2+C \|\bar v\|_1^2,
\end{align*}
\begin{align*}
-b(u_2,\bar u,A\bar w)&\leq \| u_2\|_{L^\infty} \|\bar u\|_1 \|\bar w\|_2\\
&\leq C\|\bar v\|_1\|\bar w\|_2 +C\|\bar w\|_2^2\\
&\leq C\|\bar w\|_2^2+C \|\bar v\|_1^2,
\end{align*}
and
$$
-\beta\l A^{-\vartheta}\bar u,A\bar w\r \leq -\beta\l A^{-\vartheta}\bar v,A\bar w\r
\leq C\|\bar w\|_2^2 +C \|\bar v\|_1^2.
$$
Therefore, we arrive at
$$
\ddt \E_1\leq C\E_1 + C \|\bar v\|_1^2.
$$
Recalling that $(\bar w(0),\bar \zeta^0)=(0,0)$ and exploiting \eqref{boundL12}, an application of the
Gronwall lemma provides the sought inequality
$$
\E_1(t)\leq C\int_0^t \e^{C(t-y)} \|\bar v(y)\|_1^2 \d y
\leq C \e^{Ct} \|U_{01}-U_{02}\|^2.
$$
In particular, we learn that
$$
\|\bar w(t)\|_{1} \leq C \e^{Ct}\|U_{01}-U_{02}\|_{\H},
$$
which is exactly the last point of $\textbf{(iv)}$ to be verified.
\qed

\section{Further Developments}
\label{secFD}

\noindent
In this final section, we discuss some open issues that might be the object of future investigations.

\subsection*{I}
Exploiting the techniques of this work, it is actually possible to study
generalized versions of \eqref{mezzo}, such as
$$
\partial _t\left(u + \alpha Au\right)
+\int_0^\infty  g(s) A u(t-s)\d s + \beta  h(u)
+ B(u,u)=f.
$$
for some suitable nonlinearity $h(u)$, e.g.\ the Brinkman-Forchheimer correction term
(see e.g\ \cite{STR})
$$h(u)=a|u|^p u+bu, \quad p>0,$$
where $a>0$ and $b\in \R$.

\subsection*{II}
A further interesting problem is concerned with the analysis of nonautonomous NSV equations
in presence of singularly oscillating external forces, depending on $\varepsilon>0$, of the form
$$
f^\varepsilon(t)=f_0(t)+\varepsilon^{-\rho} f_1\Big(\frac{t}{\varepsilon}\Big),\quad\rho>0,
$$
along with the formal limit obtained as $\varepsilon\to 0$,
corresponding to
$$f^0(t)=f_0(t).$$
The existence of global
attractors depending on $\varepsilon>0$ and their stability as $\varepsilon \to 0$
for NSV equations without memory has been addressed
in \cite{NSV-singosc} (see also \cite{CPV,CV} for Navier-Stokes models).
In connection with memory equations,
the techniques to handle singularly oscillating forces have been
introduced in the novel paper \cite{CCP}.

\subsection*{III}
In place of~\eqref{mezzo}
we may consider the family of equations
\begin{equation}
\label{NSVeps}
\partial _t\left(u + \alpha Au\right)
+\int_0^\infty  g_\varepsilon(s) A u(t-s)\d s + \beta  u
+ B(u,u)=f.
\end{equation}
where
$$g_\varepsilon(s)=\frac{1}{\varepsilon}g\Big(\frac{s}\varepsilon\Big),\quad \varepsilon\in(0,1],$$
is a rescaling of the original kernel $g$.
Such a rescaling has been firstly introduced in~\cite{Amnesia}.
We assume without loss of generality that
$$\int_0^\infty g(s)\d s=1.$$
In the formal limit $\varepsilon\to 0$
we have the distributional convergence
$$g_\varepsilon\to\delta_0,$$
$\delta_0$ being the Dirac mass at $0^+$. Accordingly, \eqref{NSVeps}
collapses into the NSV equation with Ekman damping
\begin{equation}
\label{NSVzero}
\partial_t  \left(u - \alpha \Delta u\right) - \Delta u
+ \beta u + B(u,u)=f.
\end{equation}
Writing as before $\mu=-g'$, and defining
$$\mu_\varepsilon(s)=\frac{1}{\varepsilon^2}\mu\Big(\frac{s}\varepsilon\Big),$$
both \eqref{NSVeps} and \eqref{NSVzero} generate dynamical systems $S_\varepsilon(t)$,
acting on their respective phase spaces
$$\H_\varepsilon=\VV\times \M_\varepsilon,$$
where
$$
\M_\varepsilon=
\begin{cases}
L^2_{\mu_\varepsilon}(\R^+;\VV)&\text{ if }\varepsilon>0,\\
\{0\}&\text{ if }\varepsilon=0.
\end{cases}
$$
Then, for all $\varepsilon\in[0,1]$, the semigroups $S_\varepsilon(t)$ possess exponential attractors
$\mathfrak{E}_\varepsilon$ on $\H_\varepsilon$. Besides, by means of the general theory
developed in~\cite{GMPZ}, one can show that the family $\{\mathfrak{E}_\varepsilon\}$
fulfills the following properties:
\begin{itemize}
\item[$\diamond$] The exponential attraction rate of $\mathfrak{E}_\varepsilon$
is uniform with respect to $\varepsilon$; namely, there exist
$\Q$ and $\omega>0$, both independent of $\varepsilon$, such that
$$
\text{dist}_{\H_\varepsilon}(S_\varepsilon(t)\mathbb{B}_{\H_\varepsilon}(R),\mathfrak{E}_\varepsilon)
\leq \Q(R)\e^{-\omega t},\quad\forall\,R\geq 0.
$$
\item[$\diamond$] The (finite) fractal dimension $\text{dim}_{\H_\varepsilon}(\mathfrak{E}_\varepsilon)$
is uniformly bounded
with respect to $\varepsilon$.
\item[$\diamond$] The family $\{\mathfrak{E}_\varepsilon\}$ is (H\"older) continuous
at $\varepsilon=0$; namely, there exist constants $C\geq 0$ and $\alpha\in(0,1)$ such that
$$
\text{dist}^{\rm sym}_{\H_{\varepsilon}}(\mathfrak{E}_\varepsilon,\mathfrak{E}_0)\leq
C \varepsilon^\alpha,
$$
where
$$
\text{dist}^{\rm sym}_{\H_{\varepsilon}}(\mathfrak{E}_\varepsilon,\mathfrak{E}_0)=
\max\big\{\text{dist}_{\H_{\varepsilon}}(\mathfrak{E}_\varepsilon,\mathfrak{E}_0),
\text{dist}_{\H_{\varepsilon}}(\mathfrak{E}_0,\mathfrak{E}_\varepsilon)\big\}
$$
is the symmetric Hausdorff distance in $\H_\varepsilon$.
\end{itemize}
A similar project has been carried out in \cite{GTM} for the equation
$$
\partial _t\left(u + \alpha Au\right)
+\nu Au+(1-\nu)\int_0^\infty  g_\varepsilon(s) A u(t-s)\d s + \beta  u
+ B(u,u)=f
$$
where $\nu\in(0,1)$ is a fixed parameter. However, in this case, the presence of the
instantaneous kinematic viscosity $\nu A u$ renders the problem easier.

\subsection*{IV}
An enhanced version of the previous analysis would be letting $\beta>0$ in \eqref{NSVeps}
to be a free parameter as well, and then considering the double limit
$$\varepsilon\to 0\qquad\text{and}\qquad
\beta\to 0.$$
In this situation, we have a family of two-parameter semigroups $S_{\varepsilon,\beta}(t)$
acting on $\H_\varepsilon$,
the limiting case $S_{0,0}(t)$ corresponding to the
NSV equation
$$
\partial_t  \left(u - \alpha \Delta u\right) - \Delta u + B(u,u)=f.
$$
For every fixed $\varepsilon>0$ and $\beta>0$, the semigroup $S_{\varepsilon,\beta}(t)$ possesses an
exponential attractor $\mathfrak{E}_{\varepsilon,\beta}$, and the same is true for the limiting
semigroup $S_{0,0}(t)$ (see~\cite{CZG}). Again, the task is proving the convergence
$$\mathfrak{E}_{\varepsilon,\beta}\to \mathfrak{E}_{0,0}$$
as $\varepsilon\to 0$ and $\beta\to 0$,
in the sense of the symmetric Hausdorff distance.
Since we are not able to provide uniform estimates when $\varepsilon>0$ and $\beta=0$ (except in the case
when $f\equiv 0$), we expect to obtain the desired result under an additional constraint on the double limit,
of the kind $\beta\geq F(\varepsilon)$, for a suitable positive function $F$ vanishing at zero.

\subsection*{Acknowledgments} This work was initiated during F.\ Di Plinio's Summer 2015 visit to the Mathematics Department at Politecnico di Milano, whose hospitality is kindly acknowledged.
F.\ Di Plinio was partially
supported by the National Science Foundation under the grants
NSF-DMS-1500449 and  NSF-DMS-1650810.
R.\ Temam was supported by NSF DMS Grant 1510249 and by the Research Fund of Indiana University.



\end{document}